\documentclass[reqno,12pt,twoside]{article}
\usepackage{fullpage}
\usepackage{amssymb}
\usepackage{amsmath}
\usepackage{amsthm}
\newtheorem{theorem}{Theorem}

\newtheorem{lemma}{Lemma}
\numberwithin{equation}{section}
\begin{document}

\title{\Large Je\'{s}manowicz' conjecture and Fermat numbers}
\author{\large Min Tang\thanks{Corresponding author. This work was supported by the National Natural Science Foundation of China, Grant
No.10901002 and Anhui Provincial Natural Science Foundation, Grant No.1208085QA02. Email: tmzzz2000@163.com} and Jian-Xin Weng }
\date{} \maketitle
 \vskip -3cm
\begin{center}
\vskip -1cm { \small
\begin{center}
School of Mathematics and Computer Science, Anhui Normal
University,
\end{center}
\begin{center}
Wuhu 241003, China
\end{center}
}
 \end{center}

 {\bf Abstract.} Let $a,b,c$ be relatively prime positive integers such that $a^{2}+b^{2}=c^{2}.$ In 1956, Je\'{s}manowicz conjectured that for any positive integer $n$, the only solution of $(an)^{x}+(bn)^{y}=(cn)^{z}$ in positive integers is $(x,y,z)=(2,2,2)$. Let $k\geq 1$ be an integer and $F_k=2^{2^k}+1$ be a Fermat number. In this paper, we show that Je\'{s}manowicz' conjecture is true for Pythagorean triples $(a,b,c)=(F_k-2,2^{2^{k-1}+1},F_k)$.



 \section{Introduction} Let $a,b,c$ be relatively prime positive integers such that $a^{2}+b^{2}=c^{2}$ with $2\mid b.$ Clearly, for any positive integer $n$, the Diophantine equation
 \begin{equation}\label{eqn1}(na)^{x}+(nb)^{y}=(nc)^{z}\end{equation}
  has the solution $(x, y, z)=(2,2,2).$ In 1956, Sierpi\'{n}ski \cite{Sierpinski} showed there is no other solution when $n=1$ and $(a,b,c)=(3,4,5)$, and Je\'{s}manowicz \cite{Jesmanowicz} proved that when $n=1$ and $(a,b,c)=(5,12,13),(7,24,25),(9,40,41),(11,60,61),$ Eq.(\ref{eqn1}) has only the solution $(x,y,z)=(2,2,2).$  Moreover, he conjectured that  for any positive integer $n,$ the Eq.(\ref{eqn1}) has no positive integer solution other than $(x,y,z)=(2,2,2).$ Let $k\geq 1$ be an integer and $F_k=2^{2^k}+1$ be a Fermat number.
  Recently, the first author of this paper and Yang \cite{Tang} proved that if $1\leq k\leq 4$, then the Diophantine equation  \begin{equation}\label{eqn2}((F_k-2)n)^{x}+(2^{2^{k-1}+1}n)^{y}=(F_kn)^{z}\end{equation}
  has no positive integer solution other than $(x,y,z)=(2,2,2)$.
  For related problems, see (\cite{Deng}, \cite{Miyazaki}, \cite{Miyazaki2}).

  In this paper, we obtain the following result.
\begin{theorem}\label{thm1} For any positive integer $n$ and Fermat number $F_k$, Eq.(\ref{eqn2}) has only the solution $(x,y,z)=(2,2,2)$.
\end{theorem}

Throughout this paper, let $m$ be a positive integer and $a$ be any integer relatively prime to $m$. If $h$ is the
least positive integer such that $a^{h}\equiv 1 \pmod m$, then $h$ is called the order of $a$ modulo $m$, denoted by $\textnormal{ord}_{m}(a)$.

\section{Lemmas}

\begin{lemma}\label{lem1}(\cite{Lu}) For any positive integer $m$, the Diophantine equation $(4m^{2}-1)^{x}+(4m)^{y}=(4m^{2}+1)^{z}$ has only the solution $(x,y,z)=(2,2,2).$\end{lemma}

\begin{lemma}\label{lem2}(See \cite[Lemma 2]{Deng}) If $z\geq max\{x,y\},$ then the Diophantine equation $a^{x}+b^{y}=c^{z},$ where $a,b$ and $c$ are any positive integers (not necessarily relative prime) such that $a^{2}+b^{2}=c^{2}$, has no solution other than  $(x,y,z)=(2,2,2).$\end{lemma}

\begin{lemma}\label{lem3} (See \cite[Corollary 1]{Le}) If the Diophantine equation $(na)^{x}+(nb)^{y}=(nc)^{z}$(with $a^2+b^2=c^2$) has a solution $(x,y,z)\neq(2,2,2),$
then $x,y,z$ are distinct.\end{lemma}

\begin{lemma}\label{lem4}(See \cite[Lemma 2.3]{Deng2013}) Let $a,b,c$ be any primitive Pythagorean triple such that the Diophantine equation $a^{x}+b^{y}=c^{z}$ has the only positive integer solution $(x,y,z)=(2,2,2)$. Then (\ref{eqn1}) has no positive integer solution satisfying $x>y>z$ or $y>x>z$.
 \end{lemma}

\begin{lemma}\label{lem5}Let $k$ be a positive integer and $F_k=2^{2^k}+1$ be a Fermat number. If $(x,y,z)$ is a solution of the Eq.(\ref{eqn2}) with $(x,y,z)\neq (2,2,2)$, then $x<z<y$.
\end{lemma}

\begin{proof} By Lemmas \ref{lem2}-\ref{lem4}, it is sufficient to prove that the Eq.(\ref{eqn2}) has no solution $(x,y,z)$ satisfying $y<z<x$.
By Lemma \ref{lem1}, we may suppose that $n\geq2$ and the Eq.(\ref{eqn2}) has a solution $(x,y,z)$ with $y<z<x$.
Then we have
 \begin{equation}\label{eqn9}2^{(2^{k-1}+1)y}=n^{z-y}\Big(F_k^{z}-(F_k-2)^{x}n^{x-z}\Big).\end{equation}
  By \eqref{eqn9} we may write $n=2^{r}$ with $r\geq1$.
  Noting that  $$\gcd\Big(F_k^{z}-(F_k-2)^{x}2^{r(x-z)},2\Big)=1,$$
  we have \begin{equation}\label{eqn10}F_k^{z}-(F_k-2)^{x}2^{r(x-z)}=1.\end{equation}
  Since $k\geq 1$, by (\ref{eqn10}) we have $F_k^z\equiv 1\pmod 3$, $z\equiv 0\pmod 2.$
  Write $z=2z_{1},$  we have
  \begin{equation}\label{eqn11}\Big(\prod\limits_{i=0}^{k-1}F_i\Big)^x2^{r(x-z)}=(F_k^{z_{1}}-1)(F_k^{z_{1}}+1).\end{equation}
    Let $F_{k-1}=\prod\limits_{i=1}^tp_i^{\alpha_i}$ be the standard prime factorization of $F_{k-1}$ with $p_1<\cdots<p_t$. By the known Fermat primes, we know that there is the possibility of $t=1$. Moreover,
    \begin{equation}\label{eqn12}\textnormal{ ord}_{p_i}(2)=2^{k}, \quad i=1,\cdots,t.\end{equation}
    Noting that $\gcd(F_k^{z_{1}}-1,F_k^{z_{1}}+1)=2,$ we know that $p_t$ divide only one of $F_k^{z_{1}}-1$ and $F_k^{z_{1}}+1$.

{\bf Case 1.} $p_t\mid F_k^{z_{1}}-1$. Then $F_k^{z_{1}}-1\equiv 2^{z_1}-1\equiv 0\pmod {p_t}$. Noting that $\textnormal{ ord}_{p_t}(2)=2^{k}$, we have $z_1\equiv 0\pmod{2^{k}}$.
By (\ref{eqn12}) we have $$F_k^{z_{1}}-1\equiv 2^{z_1}-1\equiv 0\pmod {p_i}, \quad i=1,\cdots, t.$$
Since $\gcd(F_k^{z_{1}}-1,F_k^{z_{1}}+1)=2,$ by (\ref{eqn11}) we have $$F_k^{z_{1}}-1\equiv 2^{z_1}-1\equiv 0\pmod {p_i^{\alpha_ix}}, \quad i=1,\cdots, t.$$
Hence $F_{k-1}^x\mid F_k^{z_{1}}-1$.

{\bf Case 2.} $p_t\mid F_k^{z_{1}}+1$. Then $F_k^{z_{1}}+1\equiv 2^{z_1}+1\equiv 0\pmod {p_t}$. Noting that $\textnormal{ ord}_{p_t}(2)=2^{k}$, we have $2^{k-1}\mid z_1$, but $2^{k}\nmid z_1$.
By (\ref{eqn12}) we have $$2^{2z_1}-1=(2^{z_1}+1)(2^{z_1}-1)\equiv 0\pmod {p_i}, \quad i=1,\cdots, t.$$
Thus
$$F_k^{z_{1}}+1\equiv 2^{z_1}+1\equiv 0\pmod {p_i}, \quad i=1,\cdots, t.$$
Since $\gcd(F_k^{z_{1}}-1,F_k^{z_{1}}+1)=2,$ by (\ref{eqn11}) we have $$F_k^{z_{1}}+1\equiv 2^{z_1}+1\equiv 0\pmod {p_i^{\alpha_ix}}, \quad i=1,\cdots, t.$$
Hence $F_{k-1}^x\mid F_k^{z_{1}}+1$.

However, $$F_{k-1}^x=\Big(2^{2^{k-1}}+1\Big)^x>\Big(2^{2^{k-1}}+1\Big)^{2z_1}>F_k^{z_1}+1,$$ which is impossible.

This completes the proof of Lemma \ref{lem5}.
\end{proof}

\section{Proof of Theorem \ref{thm1}}

 By Lemma \ref{lem1} and Lemma \ref{lem5}, we may suppose that $n\geq2$ and the Eq.(\ref{eqn2}) has a solution $(x,y,z)$ with $x<z<y$. Then
 \begin{equation}\label{eqn13a}\Big(\prod_{i=0}^{k-1}F_i\Big)^{x}=n^{z-x}\Big(F_k^{z}-2^{(2^{k-1}+1)y}n^{y-z}\Big).\end{equation}
It is clear from \eqref{eqn13a} that
$$\gcd\Big(n,\prod\limits_{i=0}^{k-1}F_i\Big)>1.$$
Let
$\prod\limits_{i=0}^{k-1}F_i=\prod\limits_{i=1}^{t}p_i^{\alpha_i}$
  be the standard prime factorization of $\prod\limits_{i=0}^{k-1}F_i$ and write $n=\prod\limits_{\nu=1}^{s}p_{i_\nu}^{\beta_{i_\nu}},$
   where $\beta_{i_\nu}\geq1$, $\{i_1,\cdots,i_s\}\subseteq \{1,\cdots,t\}$.  Let $T=\{1,2,\cdots, t\}\setminus \{i_1,\cdots,i_s\}$.
If $T=\emptyset$, then let $P(k,n)=1$. If $T\neq\emptyset$, then let
$$P(k,n)=\prod\limits_{i\in T}p_i^{\alpha_i}.$$
 By (\ref{eqn13a}), we have
 \begin{equation}\label{eqn14a}P(k,n)^x=F_k^{z}-2^{(2^{k-1}+1)y}\prod\limits_{\nu=1}^{s}p_{i_\nu}^{\beta_{i_\nu}(y-z)}.\end{equation}
Since $y\ge 2$, it follows that
\begin{equation}\label{eqn4.3}P(k,n)^x\equiv 1\pmod{2^{2^k}}.\end{equation}
If $3\mid P(k,n)$, then $P(k,n)\equiv -1\pmod 4$. This implies
that $x$ is even. If $3\nmid P(k,n)$, then $P(k,n)\equiv 1\pmod
4$. Let $P(k,n)=1+2^vW$, $2\nmid W$. Then $v\ge 2$. Suppose that
$x$ is odd, then
$$P(k,n)^x=1+2^vW', \quad 2\nmid W'.$$
Thus $v\ge 2^k$ and $P(k,n)\ge F_k$, a contradiction with $$
P(k,n)<\prod\limits_{i=0}^{k-1}F_i=F_k-2. $$ Therefore, $x$ is
even. Write $x=2^uN$ with $2\nmid N$. Then $u\geq 1$.

{\bf Case 1.} $P(k,n)\equiv -1\pmod 4$. Let $P(k,n)=2^dM-1$ with $2\nmid M$. Then $d\geq 2$ and
$$P(k,n)^x=1+2^{u+d}V, \quad 2\nmid V.$$
By (\ref{eqn4.3}) we have $u+d\geq 2^k$.

Choose a $\nu\in\{1,\cdots,s\}$, let $p_{i_\nu}=2^rt+1$ with $r\geq 1$, $2\nmid t$. Then
$$2^{d+r-1}<(2^dM-1)(2^rt+1)=P(k,n)\cdot p_{i_\nu}\leq \prod\limits_{i=0}^{k-1}F_i=2^{2^k}-1.$$
 Thus $d+r\leq 2^k$. Hence $u\geq r$.
 By (\ref{eqn14a}) we have \begin{equation}\label{eqn3.4}P(k,n)^x\equiv 2^z\pmod{p_{i_\nu}}.\end{equation}
Noting that $p_{i_\nu}-1\mid 2^ut$, we have \begin{equation}\label{eqn3.5}2^{tz}\equiv P(k,n)^{2^utN}\equiv 1\pmod {p_{i_\nu}}.\end{equation}
Since $\textnormal{ord}_{p_{i_\nu}}(2)$ is even and $2\nmid t$, we have $z\equiv 0\pmod 2$.

{\bf Case 2.} $P(k,n)\equiv 1\pmod 4$. Let $P(k,n)=2^{d'}M'+1$ with $2\nmid M'$. Then $d'\geq 2$ and
$$P(k,n)^x=1+2^{u+d'}V', \quad 2\nmid V'.$$
By (\ref{eqn4.3}) we have $u+d'\geq 2^k$.

Choose a $\mu\in\{1,\cdots,s\}$, let $p_{i_\mu}=2^{r'}t'+1$ with $r'\geq 1$, $2\nmid t'$. Then
$$2^{d'+r'}<(2^{d'}M'+1)(2^{r'}{t'}+1)=P(k,n)\cdot p_{i_\mu}\leq \prod\limits_{i=0}^{k-1}F_i=2^{2^k}-1.$$
 Thus $d'+r'<2^k$. Hence $u>r'$.
 By (\ref{eqn14a}) we have \begin{equation}\label{eqn3.6}P(k,n)^x\equiv 2^z\pmod{p_{i_\mu}}.\end{equation}
Noting that $p_{i_\mu}-1\mid 2^ut'$, we have \begin{equation}\label{eqn3.7}2^{t'z}\equiv P(k,n)^{2^ut'N}\equiv 1\pmod {p_{i_\mu}}.\end{equation}
Since $\textnormal{ord}_{p_{i_\mu}}(2)$ is even and $2\nmid t'$, we have $z\equiv 0\pmod 2$.

Write $z=2z_{1}, x=2x_{1}$.
By (\ref{eqn14a}), we have \begin{equation}\label{eqn15T}2^{(2^{k-1}+1)y}\prod\limits_{\nu=1}^{s}p_{i_\nu}^{\beta_{i_\nu}(y-z)}=\Big(F_k^{z_{1}}-P(k,n)^{x_1}\Big)\Big(F_k^{z_{1}}+P(k,n)^{x_1}\Big).\end{equation}
Noting that $$\gcd\Big(F_k^{z_{1}}-P(k,n)^{x_1},F_k^{z_{1}}+P(k,n)^{x_1}\Big)=2,$$
we have \begin{equation}\label{eqn16T}2^{(2^{k-1}+1)y-1}\mid F_k^{z_{1}}-P(k,n)^{x_1},\quad 2\mid F_k^{z_{1}}+P(k,n)^{x_1},\end{equation}
 or \begin{equation}\label{eqn17T}2\mid F_k^{z_{1}}+P(k,n)^{x_1},\quad 2^{(2^{k-1}+1)y-1}\mid F_k^{z_{1}}-P(k,n)^{x_1}.\end{equation}
However, $$2^{(2^{k-1}+1)y-1}>2^{(2^{k-1}+1)2z_1}>(F_k+F_k-2)^{z_1}>F_k^{z_{1}}+P(k,n)^{x_1},$$
a contradiction.

This completes the proof of Theorem \ref{thm1}.

\section{Acknowledgment}  We sincerely thank Professor Yong-Gao Chen for his
valuable suggestions and useful discussions. We would like to thank the referee for his/her helpful comments.

\end{document}